\documentclass[12pt]{amsart}
\usepackage{amssymb,latexsym}
\usepackage{pstcol,pstricks,color}

\advance\textheight by \topskip   \numberwithin{equation}{section}

\newtheorem{theorem}{Theorem}[section]

\newtheorem{corollary}[theorem]{Corollary}

\begin{document}
\title[Bayes Blind Spot]{A Further Look at the Bayes Blind Spot}

\author[M. Shattuck]{Mark Shattuck}
\address{Department of Mathematics, University of Tennessee,
37996 Knoxville, TN, USA}
\email{mshattuc@utk.edu}

\author[C. Wagner]{Carl Wagner}
\address{Department of Mathematics, University of Tennessee,
37996 Knoxville, TN, USA}
\email{cwagner@tennessee.edu}

\maketitle
\thispagestyle{empty}

\begin{abstract}
\footnotesize{Gyenis and R\'{e}dei have demonstrated that any prior $p$ on a finite algebra, however chosen, severely restricts the set of posteriors accessible from $p$ by Jeffrey conditioning on a nontrivial partition.  Their demonstration involves showing that the set of posteriors not accessible from $p$ in this way (which they call the \emph{Bayes blind spot of} $p$) is large with respect to three common measures of size, namely, having cardinality $c$, (normalized) Lebesgue measure $1$, and Baire second category with respect to a natural topology.  In the present paper, we establish analogous results for probability measures defined on any infinite sigma algebra of subsets of a denumerably infinite set.  However, we have needed to employ distinctly different approaches to determine the cardinality, and especially, the topological and measure-theoretic sizes of the Bayes blind spot in the infinite case.  Interestingly, all of the results that we establish for a single prior $p$ continue to hold for the intersection of the Bayes blind spots of countably many priors.  This leads us to conjecture that Bayesian learning itself might be just as culpable as the limitations imposed by priors in enabling the existence of large Bayes blind spots.}\normalsize
\end{abstract}

\footnotesize{\emph{Keywords:} Bayes blind spot, Jeffrey conditioning, Baire category, Bayesian learning }\normalsize

\section{Introduction}

In a pair of groundbreaking papers, Gyenis and  R\'{e}dei (2017, 2021) initiated a study of the following important epistemological question: Given a prior probability measure $p$, and a family $R$ of possible methods for revising $p$ in response to new evidence, how large is the set of probability measures $q$ that are \emph{inaccessible} from $p$ using any of the stipulated revision methods?  The authors call the set of probability measures that are inaccessible from $p$ its \emph{Bayes blind spot}, denoted by $BS(p)$.  In their foundational 2017 paper, Gyenis and R\'{e}dei (henceforth, G\&R) situate the accessibility problem in the broadly comprehensive setting in which revision is conceptualized in terms of conditional expectation.  In their 2021 paper, they explore in detail the case in which probability measures are defined on a \emph{finite} algebra $\bf{A}$ of subsets of an arbitrary set $Y$,$^1$ and $R$ consists of all possible ways of revising $p$ by Jeffrey conditioning on any (necessarily finite) partition of $Y$ consisting of members of $\bf{A}$, at least one of which has at least two elements.  They prove in this case that the Bayes blind spot is large with respect to three common measures of size, namely, having cardinality $c$, (normalized) Lebesgue measure $1$, and Baire second category with respect to a natural topology.  They observe that their result emphasizes the heretofore insufficiently recognized crucial role of priors in Bayesian learning, and demonstrates that any prior, however chosen, severely restricts the set of posteriors that are in principle derivable from that prior by a single application of Jeffrey conditioning.$^2$

In the present paper, we establish analogous results for probability measures defined on any infinite sigma algebra of subsets of a denumerably infinite set $Y$.$^3$  However, we have needed to employ distinctly different approaches to determine the cardinality, and the topological and measure-theoretic sizes of the Bayes blind spot in the infinite case.  Interestingly, all of the results that we establish for a single prior $p$ continue to hold for the intersection of the Bayes blind spots of countably many priors.  This leads us to conjecture that Bayesian learning itself might be just as culpable as the limitations imposed by priors in enabling the existence of large Bayes blind spots.  Before presenting our results, however, we offer a brief review of some essential features of Jeffrey conditioning.

\section{A primer on Jeffrey conditioning}

Suppose $X$ is \emph{any} set of possible states of the world, ${\bf A}$ is a sigma algebra of subsets of $X$, and ${\bf E}=\{E_i\}$ is a countable partition of $X$ with each $E_i \in {\bf A}$. It will be convenient throughout this paper to assume that every prior $p$ is a \emph{strictly coherent} probability measure on ${\bf A}$, in the sense that $p(A)>0$ for all nonempty $A \in \bf{A}$.$^4$ A probability measure $q$ on $\bf{A}$ is said to come from $p$ by \emph{Jeffrey conditioning} (henceforth, JC) on $\textbf{E}$ if
\begin{equation}\label{JC}
q(A)=\sum_iq(E_i)p(A|E_i), \quad \text{for all } A \in \bf{A}.
\end{equation}

\begin{theorem}\label{thJC1}
Formula \eqref{JC} is equivalent to the rigidity condition
\begin{equation}\label{JC2}
q(A|E_i)=p(A|E_i), \text{ for each } i \text { such that } q(E_i)>0 \text{ and all } A\in {\bf A},
\end{equation}
and, if $X$ is countable and ${\bf A}=2^{X}$, to the following identities:
\begin{equation}\label{JC3}
\text{For each } i \text{ and all } x,x' \in E_i,~\frac{q(x)}{p(x)}=\frac{q(x')}{p(x')},
\end{equation}
where we slightly abuse notation by writing, for example, $q(x)$, instead of $q(\{x\})$.
\end{theorem}
\begin{proof}
See Jeffrey (1992, pp.~117--118).
\end{proof}

The partition $\{\{x\}:x\in X\}$, in which each block is a singleton, is called the \emph{trivial partition of} $X$.  All other partitions are \emph{nontrivial}. It is easy to see that if $X$ is countable and $p$ is strictly coherent on $2^X$, then \emph{every} probability measure $q$ on $2^X$ comes from $p$ by JC on the trivial partition of $X$.
In this case, JC amounts to \emph{total reassessment}, which is why we restrict consideration to JC on nontrivial partitions.

The following theorem is an immediate consequence of Theorem \ref{thJC1}.

\begin{theorem}\label{thJC2}
If $X$ is countable, $p$ and $q$ are probability measures on $2^X$, and $p$ is strictly coherent, then $q$ comes from $p$ by JC on a nontrivial partition $\bf{E}$ of $X$ if and only if there exist distinct $x$ and $x'$ in $X$ such that $q(x)/p(x)=q(x')/p(x')$.
\end{theorem}
\begin{proof}
Necessity. Since at least one block of a nontrivial partition contains at least two elements of $X$, this follows from \eqref{JC3} above.  Sufficiency. By assumption, the map $x \mapsto q(x)/p(x)$ from $X$ to $[0,\infty)$ is not injective.  For each member $r$ in the range $R$ of this map, let $E_r:=\{x\in X:q(x)/p(x)=r\}$.  Clearly, ${\bf E}:=\{E_r:r \in R\}$ is a countable, nontrivial partition of $X$, and for each $r \in R$, if $x,x' \in E_r$, then $q(x)/p(x)=q(x')/p(x')=r$.  So by Theorem \ref{thJC1}, $q$ comes from $p$ by JC.
\end{proof}

\emph{Remark:} The partition ${\bf E}$ introduced above in the proof of sufficiency is the \emph{coarsest partition}$^5$ of $X$ on which $q$ comes from $p$ by JC.  See Diaconis and Zabell (1982, 824) and van Fraassen (1980). \medskip

From Theorem \ref{thJC2}, the following characterization of $BS(p)$ is immediate.

\begin{corollary}\label{corJC}
The probability measure $q$ belongs to $BS(p)$ if and only if the ratios $q(x)/p(x)$ are distinct for all $x \in X$.
\end{corollary}

In the following sections, the above corollary will be invoked repeatedly in determining the size of $BS(p)$ in terms of its cardinality, its Baire category and an appropriately constructed measure.

\section{The cardinality of $BS(p)$}

Let $X=\{x_i\}_{i\geq1}$ be a countable set with $|X|\geq 2$. Since all probability measures $q$ on $2^X$ are completely determined by the values $q_i:=q(x_i)$, we may focus attention on the associated probability distributions (i.e., mass functions), denoted by ${\bf q}=(q_i)_{i\geq1}$.  We prove here a more general result concerning the cardinality of what we now denote by $BS(\bf{p})$.  Given a nonempty set ${\bf P}$ of probability distributions on $X$, let $BS({\bf P})$ be the Bayes blind spot of ${\bf P}$ defined as the intersection of $BS({\bf p})$ for all ${\bf p} \in {\bf P}$.

\begin{theorem}\label{Cardth}
Let ${\bf P}$ be a nonempty countable set of probability distributions on a countable set $X$ with $|X|\geq 2$, where all probabilities in each distribution are positive.  Then the cardinality of $BS({\bf P})$ is equal to $c$.
\end{theorem}
\begin{proof}
We treat only the case when $X$ is denumerably infinite, as the simple modifications required to accommodate the finite case of $X$ will be apparent.
We show first that $BS({\bf P})$ is nonempty.  Let ${\bf P}=\{{\bf p}^{(1)},{\bf p}^{(2)},\ldots\}$, where ${\bf p}^{(k)}=(p_i^{(k)})_{i\geq 1}$ with $\sum_{i\geq 1}p_i^{(k)}=1$ and $p_i^{(k)}>0$ for all $i$, for each $k \geq 1$.  Let $m_1$ be any positive real number. If $i \geq 2$, and the positive real numbers $m_1,\ldots,m_{i-1}$ have been chosen, choose $m_i$ such that (i) $0<m_i<2^{-i}$ and (ii) $\frac{m_i}{p_i^{(k)}}\neq \frac{m_{j}}{p_{j}^{(k)}}$ for all $1 \leq j \leq i-1$ and $k \geq 1$.  Note that such a sequence $(m_i)_{i\geq 1}$ can be constructed, since at each step, only a countable number of possible values for $m_i$ are being excluded from the (uncountable) set $(0,2^{-i})$.  Further, the series $\sum_{i\geq 1}m_i$ is convergent, with, say, $\sum_{i\geq 1}m_i=m$.  Define ${\bf q}=(q_i)_{i\geq1}$ by setting $q_i=\frac{m_i}{m}$.  Then the sequence $\frac{q_i}{p_i^{(k)}}$ for $i \geq 1$ has all its terms distinct for each $k \geq 1$, by construction, with $\sum_{i\geq 1}q_i=1$, so ${\bf q} \in BS({\bf P})$.

Let $\epsilon=\min\{1-q_1,q_2\}$.  Given $0<\delta<\epsilon$, let $s_\delta^{(k)}$ be the sequence defined by
$\frac{q_1+\delta}{p_1^{(k)}},\frac{q_2-\delta}{p_2^{(k)}},\frac{q_3}{p_3^{(k)}},\frac{q_4}{p_4^{(k)}},\ldots$ for each $k \geq 1$.  Note that ${\bf q} \in BS({\bf P})$ implies that there are only a countable number of values of $\delta$ for which $s_\delta^{(j)}$ fails to have all its terms distinct for some $j \geq 1$.  Then for all other values of $\delta$ in the interval $(0,\epsilon)$, we have $q_\delta=(q_1+\delta,q_2-\delta,q_3,q_4,\ldots) \in BS({\bf P})$.  Hence, $|BS({\bf p})| \geq c$, as it contains an uncountable subset.  On the other hand, $BS({\bf p})\subset [0,1]^{\mathbb{P}}$, where $\mathbb{P}=\{1,2,\ldots\}$, and so $|BS({\bf p})| \leq |[0,1]^{\mathbb{P}}|=c$, whence $|BS({\bf p})|=c$.
\end{proof}

\section{A topological analysis of $BS(p)$}

In this section and the next, $X$ represents a denumerably infinite set.  Given a real sequence ${\bf s}=(s_i)_{i\geq 1}$, the $\ell_1$-norm of ${\bf s}$ is defined as $\sum_{i\geq 1}|s_i|$, with the distance between $\bf{s}$ and ${\bf t}=(t_i)_{i\geq 1}$ given by $|{\bf s}-{\bf t}|=\sum_{i\geq 1}|s_i-t_i|$. A fundamental result from analysis states that the subset of the real sequences for which the $\ell^1$-norm is finite (i.e., $\sum_{i\geq 1}|s_i|$ converges) is a complete metric space (with metric $d$ defined as $d({\bf s},{\bf t})=|\bf{s}-\bf{t}|$); see, e.g., Friedman (1982, Theorem 3.2.3).  Let $S$ denote the set of sequences summing to $1$ with nonnegative entries, which are synonymous with the probability distributions on $X$.  Then $S$ itself is complete (in the topology induced by the $\ell^1$-norm), being a closed subset of a complete metric space.

In this (and the following) section, let ${\bf p}=(p_i)_{i \geq 1}$ denote a member of $S$ where $p_i>0$ for all $i$.  We shall make use of the characterization of $BS({\bf p})$ supplied by Corollary \ref{corJC} as consisting of those ${\bf q}=(q_i)_{i \geq 1}$ in $S$ such that the $\frac{q_i}{p_i}$ values for $i \geq 1$ are all distinct from one another.

Then $BS({\bf p})$ satisfies the following topological properties.

\begin{theorem}\label{th1}
Let ${\bf p}=(p_i)_{i\geq 1} \in S$, with $p_i>0$ for all $i$.  Then $BS({\bf p})$ in the $\ell^1$-norm topology on $S$ (i) is of second category, (ii) is dense in $S$, and (iii) has an empty interior.
\end{theorem}
\begin{proof}
(i) Given $1 \leq i <j$, let $S_{i,j}$ denote the subset consisting of those members ${\bf q}=(q_i)_{i \geq 1}$ of $S$ such that $\frac{q_i}{p_i}=\frac{q_j}{p_j}$.  Note that $S_{i,j}$ is a closed subset of $S$ having an empty interior, as one can clearly find ${\bf v} \in S-S_{i,j}$ with $|{\bf v}-{\bf q}|<\epsilon$ for any given $\epsilon >0$ and ${\bf q}$ (upon perturbing the $q_i$ and $q_j$ entries of ${\bf q}$ slightly).  Since the closure of $S_{i,j}$ has empty interior, it is nowhere dense for each $i$ and $j$ and thus
$$BS({\bf p})^c=\bigcup_{1\leq i<j}S_{i,j}$$
is a countable union of nowhere dense sets. By definition, $BS({\bf p})^c$ is of first category.  Now $S$ complete implies it is of the second category, by the Baire category theorem, see Friedman (1982, p.~106).  But then $S=BS({\bf p})\cup BS({\bf p})^c$ implies $BS({\bf p})$ must be of the second category, for otherwise $BS({\bf p})\cup BS({\bf p})^c$ would be of first category being a countable union of nowhere dense sets. \medskip

(ii) Let ${\bf q} \in S$.  We will find a member of $BS({\bf p})$ whose $\ell_1$-distance from ${\bf q}$ is arbitrarily small. Let $0<\epsilon<\frac{1}{2}$ be given.  Define a sequence ${\bf r}=(r_1,r_2,\ldots)$ of nonnegative real numbers as follows.  Let $r_1$ be given and for $n>1$, let $r_n \geq 0$ be such that $\frac{r_n}{p_n}$ is distinct from all values of $\frac{r_i}{p_i}$ for $1 \leq i \leq n-1$.  Further, we may assume $|r_n-q_n|<\frac{\epsilon}{2^n}$ for all $n \geq 1$.  Then ${\bf r}$ is a convergent series with
$$|{\bf r}-{\bf q}|<\sum_{n\geq 1}\frac{\epsilon}{2^n}=\epsilon,$$
and hence $||{\bf q}|-|{\bf r}||<\epsilon$ implies $\frac{1}{2}<1-\epsilon<|{\bf r}|<1+\epsilon$.  Let $r_n'=\frac{r_n}{|{\bf r}|}$ and ${\bf r'}=(r_n')_{n\geq 1}$.  Then ${\bf r'} \in BS({\bf p})$ and
\begin{align*}
|{\bf r'}-{\bf q}|&=\frac{1}{|{\bf r}|}|{\bf r}-|{\bf r}|{\bf q}|\leq 2|{\bf r}-|\bf{r}|{\bf q}| \\
&\leq2\left(|{\bf r}-{\bf q}|+|1-|\bf {r}||\cdot|{\bf q}|\right)<2(\epsilon+\epsilon\cdot1)=4\epsilon,
\end{align*}
whence $BS({\bf p})$ is dense in $S$. \medskip

(iii) Let ${\bf q}=(q_i)_{i \geq 1} \in BS({\bf p})$.  We will find ${\bf q'}=(q_i')_{i\geq1}\in S-BS({\bf p})$ such that $|{\bf q}-{\bf q'}|$ is arbitrarily small.  Let $0<\epsilon<q_2$, where we assume for now $q_2>0$.  Let $n \geq 1$ be large enough so that $\max\{q_i,\frac{p_i}{p_1}\}<\epsilon$ for all $i \geq n$.  First assume $\frac{q_n}{p_n}-\frac{q_1}{p_1}>0$ and let $\delta=q_n-\frac{p_nq_1}{p_1}>0$.  Let the $(q_i')_{i \geq 1}$ be given by $q_i'=q_i$ if $i\neq n,n+1$, with $q_n'=\frac{q_1p_n}{p_1}$ and $q_{n+1}'=q_{n+1}+\delta$.  One may verify ${\bf q'}\in S-BS({\bf p})$.  Then $\frac{p_nq_1}{p_1} \leq \epsilon q_1<\epsilon$ implies $|q_n-q_n'|<\epsilon$, being the difference of two nonnegative real numbers less than $\epsilon$, and further $|q_{n+1}-q_{n+1}'|=\delta<\epsilon$.  Thus, we get $|{\bf q}-{\bf q'}|=|q_n-q_n'|+|q_{n+1}-q_{n+1}'|<2\epsilon$.

Now assume $\frac{q_n}{p_n}-\frac{q_1}{p_1}<0$.  Let $\rho=\frac{p_nq_1}{p_1}-q_n>0$ and note $\rho \leq \frac{p_nq_1}{p_1}\leq \epsilon q_1<q_2$.  Define the $q_i'$ in this case by $q_i'=q_i$ if $i \neq 2,n$, with $q_2'=q_2-\rho$ and $q_n'=\frac{p_nq_1}{p_1}$ (note $n>2$ since $\epsilon<q_2$, by assumption).  Then we have $|q_2-q_2'|,|q_n-q_n'|<\epsilon$ and thus $|{\bf q}-{\bf q'}|<2\epsilon$.  On the other hand, if $q_2=0$, then ${\bf q} \in BS({\bf p})$ implies $q_3>0$, and one may proceed similarly as before with $q_3$ in place of $q_2$, upon requiring $0<\epsilon<q_3$.  This implies in all cases that ${\bf q}$ cannot be an interior point of $BS({\bf p})$, whence $BS({\bf p})$ has an empty interior.
\end{proof}

The preceding result may be extended as follows.

\begin{theorem}\label{th2}
The blind spot $BS({\bf P})$, where ${\bf P}$ is a nonempty countable subset of $S$, satisfies the same conditions (i)--(iii) given above with respect to the $\ell^1$-norm as $BS({\bf p})$.
\end{theorem}
\begin{proof}
We make simple modifications to the preceding proof where required.  For property (i) concerning $BS({\bf P})$, note that the countable intersection of second category sets each of whose complement is of first category is of second category.  For (ii),  observe that the same proof applies when there are a countable number of ${\bf p}$  since for each $n>1$, one has an interval of potential values for $r_n$ wherein at most a denumerably infinite number of values are excluded as possibilities.  Property (iii) follows from the fact that the interior operator respects subset inclusions.
\end{proof}

\emph{Discussion:} Due to (iii), Theorem \ref{th1} demonstrates that the infinite case of $X$ differs fundamentally from the earlier comparable result from G\&R (2021) when $X$ was finite where it was shown that $BS({\bf p})$ is an open dense subset in the set of all probability distributions on $X$.   For a more basic example of a subset $T$ in a metric space $Y$ such that $T$ is dense in $Y$ and of second category, but yet has an empty interior, consider the subset of irrationals within the set of reals.  Further, it is seen that $BS({\bf p})$ is neither open nor closed in the infinite case, as both $BS({\bf p})$ and $BS({\bf p})^c$ have empty interior.

Recall that the $\ell^p$-norm for each $p \geq 1$ is defined for ${\bf s}=(s_i)_{i\geq1}$ as $\left(\sum_{i\geq 1}|s_i|^p\right)^{\frac{1}{p}}$, with the limiting case as $p \rightarrow \infty$ denoted by $\ell^\infty$ given by $\max\{|s_i|:i\geq1\}$.  Since the topology of the $\ell^p$-norm strictly refines that of the $\ell^q$-norm for $1 \leq p<q \leq \infty$, with the $\ell^p$-metric complete for each $p \geq 1$, the prior two theorems apply also to all of the $\ell^p$-norms on $S$.  Further, the results are seen to apply to the complete metric $d$ defined by $d({\bf u},{\bf v})=\frac{|{\bf u}-{\bf v}|}{1+|{\bf u}-{\bf v}|}$ for ${\bf u},{\bf v}\in S$, where $|\cdots|$ denotes an $\ell^p$-norm.  Note that since the topology on $S$ induced by $d$ is bounded, it is seen not to be equivalent to the topology induced by any of the $\ell^p$-norms.  It would be interesting to prove analogues of Theorems \ref{th1} and \ref{th2} for other kinds of topologies on $S$ by considering a greater variety of metrics on $S$.

Finally, the result in (iii) from Theorem \ref{th1} can be generalized in another way as follows.  Consider the subset $S_\ell$ of $S$ where $\ell \geq 1$ consisting of those distributions ${\bf q}=(q_i)_{i\geq1}$ in which there are at least $\ell$ pairs $(n,m)$ where $1 \leq n <m$ such that $\frac{q_n}{p_n}=\frac{q_m}{p_m}$.  Then the set $S_\ell$ is dense in $S$ for each $\ell\geq 1$, the details of the proof we leave to the reader. Note that the $\ell=1$ case corresponds to (iii). \medskip

\emph{Open question:} To what degree can the results in Theorems \ref{th1} and \ref{th2} be extended beyond the class of topologies on $S$ corresponding to the $\ell^p$-norms?  In particular, is it possible to find general sufficient (more desirably, necessary and sufficient) conditions for a complete metric $d$ on $S$ which would ensure $BS(\bf{p})$ is of second category in $(S,d)$ for all ${\bf p}$?

\section{A measure-theoretic analysis of $BS(p)$}

Let $S$ denote the set of all probability distributions on a countably infinite set $X$ whose associated mass functions are represented sequentially.

\begin{theorem}\label{measureth}
There exists a probability measure on $S$ such that $BS({\bf p})$ has probability $1$ with respect to this measure for each ${\bf p}=(p_i)_{i\geq1} \in S$ having all positive components.
\end{theorem}
\begin{proof}
Let $X_1$ be a uniform random variable on the interval $[0,1)$.  Define the random variables $X_i$ for $i>1$ recursively by letting $X_i$ be uniform on the interval $[0,1-\sum_{j=1}^{i-1}x_j)$, where $X_j=x_j$ for $1 \leq j \leq i-1$.  Then $X_1+X_2+\cdots \rightarrow 1$ almost surely, and we consider the set of possible outcomes $(x_1,x_2,\ldots)$ wherein $X_i=x_i$ for all $i$, which are synonymous with the members of the set $S$.  Note that outcomes $(x_1,x_2,\ldots)$ which are finitely nonzero have probability zero of occurring and hence the subset comprising the corresponding members of $S$ has measure zero.

Given integers $1 \leq a <b$, consider $P(X_a=cX_b)$, where $c=c_{a,b}$ is the constant defined by $c=\frac{p_a}{p_b}$.  Then we have

\begin{align*}
P(X_a=cX_b)=\int\int &P(X_a=ct|\sum_{i=1}^{a-1}X_i=s \text{ and } X_b=t)\\
&\quad\times(\text{Joint density of } \sum_{i=1}^{a-1}X_i \text{ and } X_b \text{ evaluated at } s \text{ and } t) dsdt.
\end{align*}

Note that the probability $P(X_a=ct|...)$ in the prior integral is zero since for each $s$ and $t$, it can be shown that $X_a$ has a (conditional) density, whence $P(X_a=cX_b)=0$.  Considering all pairs $(a,b)$ where $a<b$ (a countable number of possibilities) implies the probability that some $(x_1,x_2,\ldots)$ is not in $BS({\bf p})$ is zero. Thus, $BS({\bf p})$ has probability $1$ in this measure on $S$.
\end{proof}

\emph{Discussion:}  Since $BS({\bf p})^c$ has probability $0$ with respect to the measure $M$ defined on $S$ in the preceding proof, then so does $BS({\bf P})^c$, where ${\bf P}$ consists of a countable number of distributions $\bf{p}$.  Hence, $BS({\bf P})$ has probability $1$ with respect to $M$.

Note that there is apparently not a straightforward extension to the infinite case of the argument from G\&R (2021) in the finite case of $X$ which demonstrated that the complement of $BS({\bf p})$ has measure zero with respect to the Lebesgue measure on $\mathbb{R}^{n-1}$, where $n=|X|$.  This is due in part to the fact that there is apparently not an analogous measure on $\mathbb{R}^\infty$ that permits a comparable analysis.  It should be remarked that the proof of Theorem \ref{measureth} can be extended to the finite case of $X$ by setting $X_n=1-\sum_{i=1}^{n-1}x_i$ and terminating the recursive procedure of defining random variables.  Note however that the resulting measure for the probability distributions on $X$ differs from that supplied by (normalized) Lebesgue measure on $\mathbb{R}^{n-1}$.

Further, different measures on $S$ for which $BS(\bf{p})$ has probability $1$ can be obtained by allowing the $X_i$ to assume other continuous distributions on a finite interval.  For example, one could let $X_1$ be a $\beta$ distribution on $[0,1)$ and then subsequently define the $X_i$ for $i>1$ as appropriate scaled versions of the $X_1$ distribution on intervals of decreasing length.  Note that it is not a requirement that the $X_i$ all have the same kind of distribution, provided they are continuous and $\sum_{i\geq 1}X_i$ converges to $1$ almost surely. \medskip

\emph{Open Question:} Is it possible to find a general criterion for a probability measure on $S$ to ensure that $BS(\bf{p})$ has probability $1$ for all possible $\bf{p}$?

\section{Summary and conclusion}

We initiated this project to determine if Jeffrey conditioning on partitions of a denumerably infinite set might offer a degree of flexibility sufficient to ameliorate the limitative results established by G\&R in the finite case. On a superficial level, the answer turns out to be negative.  For the Bayes blind spot of a prior distribution on a denumerably infinite set $X$ has cardinality $c$, is of second Baire category in the set $S$ of all probability distributions on $X$ for a natural topology on $S$, and has measure 1 for an appropriately defined measure on $S$.  Our proof that $|BS({\bf p})|=c$ holds for all countable $X$ and, in the case when $X$ is finite, differs from, though it is in the same spirit as, the proof given by G\&R.  Further, the proof of G\&R may be extended to all countable $X$, and applies also to show $|BS({\bf P})|=c$ in the finite, though not the infinite, case of ${\bf P}$.  However, our approach to the measure-theoretic and topological sizes of $BS(\bf{p})$ when $X$ is denumerably infinite differs more dramatically from that employed by G\&R in the finite case, as detailed below.

The difference is particularly striking in the topological analysis, where we show that $BS(\bf{p})$ is (i) of second category, (ii) is dense in $S$, and (iii) \emph{has an empty interior}, with the topology on $S$ derived from its $\ell^1$-norm (or any $\ell^p$-norm, where $1 \leq p \leq \infty$).  Indeed, it is precisely the flexibility afforded by probabilities defined on a denumerably infinite set that allows us, for any ${\bf q} \in BS({\bf p})$, to find a member of $S-BS({\bf p})$ whose $\ell^1$-distance from ${\bf q}$ is arbitrarily small.  So this topology does not reduce to the topology employed by G\&R in the finite case, where $BS({\bf p})$ was shown to be an \emph{open}, dense subset of $S$.  If, as we suspect, separate treatments of the finite and denumerably infinite cases are unavoidable here, this represents a somewhat unusual situation in mathematics, where one typically expects a unified treatment under the rubric of countability of the relevant underlying sets.

In our measure-theoretic analysis, we constructed a measure $M$ on the set of probability distributions on the denumerably infinite set $X$ such that $M(BS({\bf p}))=1$.  Moreover, a slight variation on this construction when $X$ is finite produces a measure $m$ such that $m(BS({\bf p}))=1$.  So we have here a more or less unified treatment for all countable sets $X$.  This measure $m$ differs, however, from the (normalized) Lebesgue measure employed by G\&R in their approach to the finite case.

It is particularly noteworthy that if ${\bf P}$ is any countable family of priors on a denumerably infinite set and $BS({\bf P}):=\bigcap_{{\bf p} \in {\bf P}}BS({\bf p})$, then $BS({\bf P})$ continues to have the topological and measure-theoretic sizes of the Bayes blind spot of a single prior, as described in the preceding two paragraphs, and has the same cardinality as well.  In particular, (1) $|BS({\bf P})|=c$, (2) $BS({\bf P})$ is of second category, is dense in $S$, and has empty interior for the topology on $S$ derived from its $\ell^1$-norm, and (3) $M(BS({\bf P}))=1$.  This raises an intriguing question, which we have been unable to resolve, and with which we conclude this paper:

Do these results strengthen the observation of G\&R that any prior, however chosen, puts severe limitations on Bayesian learning originating in that prior (after all, the intersection of the Bayes blind spots of countably many Bayesian agents, each with a different prior, and each employing Jeffrey conditioning, remains topologically and measure-theoretically large)?  Or do they suggest that \emph{Bayesian learning itself} might be equally culpable in allowing the existence of large Bayes blind spots (even with access to countably many priors, the union of all the probabilities accessible from these priors by Jeffrey conditioning remains topologically and measure-theoretically small)?

\section{Notes}

1. Actually, Gyenis and R\'{e}dei treat the case in which this algebra is $2^X$, where $X$ is a finite set.  This involves no loss of generality, since any finite algebra $\bf{A}$ of subsets of an arbitrary set is isomorphic to $2^X$ for some finite set $X$, namely, the set of atoms of $\bf{A}$ (R\'{e}nyi, 1970, Theorem 1.6.1).  It may be worth recalling here that a finite algebra of subsets of an arbitrary set is \emph{ipso facto} a sigma algebra, since every countable union of such subsets is equal to some finite union of those subsets.  Similarly, every finitely additive probability measure on a finite algebra is countably additive since, in every infinite sequence of pairwise disjoint sets from that algebra, all but finitely many are equal to the empty set. \medskip

2. Gyenis and  R\'{e}dei (2021, section 5.1) also analyze the Bayes blind spot from the perspective of repeated applications of Jeffrey conditioning. \medskip

3. In fact, we are able, with no loss of generality, to restrict consideration to probability measures defined on $2^Y$, where $Y$ is denumerably infinite.  See R\'{e}nyi (1970, Theorem 1.6.2). \medskip

4. This simplifying assumption is made to avoid continually having to specify in various definitions and theorems that certain probabilities are nonzero.  Recall that if probabilities are construed as the (linear utility) prices one is willing to pay for certain bets, then confirming your probabilities with the axioms of finitely additive probability (so-called \emph{coherent} probabilities) protects you against accepting a finite sequence of bets on which you are sure to sustain a net loss.  It is commonly believed that probability measures must be strictly coherent in order to avoid accepting bets on which a net gain is impossible, but a net loss is possible (a so-called \emph{weak Dutch book}). But see Wagner (2007) regarding a slightly modified conception of subjective probability in which mere additivity suffices not only to protect against sure loss, but also against vulnerability to a weak Dutch book.  \medskip

5. If ${\bf E}$ and ${\bf F}$ are partitions of a set $X$, then ${\bf F}$ is \emph{coarser than} ${\bf E}$ if, for every $E \in {\bf E}$, there exists an $F \in {\bf F}$ such that $E\subseteq F$ (equivalently, if every $F \in {\bf F}$ is a union of members of ${\bf E}$).

\end{document}